\documentclass[11pt]{amsart}

\usepackage{times,amsmath,amsthm,amsfonts,amssymb}
\usepackage{mathrsfs}
\usepackage{fullpage}
\usepackage{graphicx, color}

\title[$L^{(\alpha)}$-multiplier Sequences]{Multiplier Sequences for Generalized Laguerre Bases}

\author{Tam\' as Forg\' acs}
\address{Department of Mathematics\newline \indent
California State University, Fresno, 93740}
\email{{\tt tforgacs@csufresno.edu}}

\author{Andrzej Piotrowski}
\address{Department of Natural Sciences\newline\indent
University of Alaska Southeast, Juneau, 99801}
\email{{\tt apiotrowski@uas.alaska.edu}}

\date{}

\newcommand{\N}{{\mathbb N}}

\newcommand{\R}{{\mathbb R}}

\newcommand{\set}[1]{\left\{ #1 \right\}}

\newcommand{\mbb}{\mathbb}
\newcommand{\ds}{\displaystyle}

\newcommand{\seq}{\ds \{\gamma_k\}_{k=0}^{\infty}}

\begin{document}
\begin{abstract}
In this paper we present a complete characterization of geometric and linear $L^{(\alpha)}$-multiplier sequences. In addition, we give a partial characterization of the generic $L^{(\alpha)}$-multiplier sequence, and pose some open questions regarding polynomial type $L^{(\alpha)}$-multiplier sequences.
\end{abstract}
\maketitle

\theoremstyle{plain}
\newtheorem{thm}{\sc Theorem}[section]
\newtheorem{lem}[thm]{\sc Lemma}
\newtheorem{o-thm}[thm]{\sc Theorem}
\newtheorem{prop}[thm]{\sc Proposition}
\newtheorem{cor}[thm]{\sc Corollary}

\theoremstyle{definition}
\newtheorem{conj}[thm]{\sc Conjecture}
\newtheorem{defn}[thm]{\sc Definition}
\newtheorem{qn}[thm]{\sc Question}
\newtheorem{ex}[thm]{\sc Example}
\newtheorem{pr}[thm]{\sc Problem}

\theoremstyle{remark}
\newtheorem*{rmk}{\sc Remark}
\newtheorem*{ack}{\sc Acknowledgment}

\section{Introduction}
Corresponding to any sequence of real numbers $\seq$ one can define a linear operator $T$ on $\mbb{R}[x]$ by declaring $T[x^n] = \gamma_n x^n$ for all $n$.  If the linear operator $T$ has the property that $T[p]$ has only real zeros whenever $p$ has only real zeros, then $\{\gamma_k\}_{k=0}^{\infty}$ is called a \textit{multiplier sequence}.  Examples of such sequences were demonstrated in the late 1800's by Jensen and Laguerre and, in the early 1900's, all such sequences were completely characterized by P\'olya and Schur.  
\begin{thm}\label{PS}
\emph{(P\'olya-Schur \cite{PS})} Let $\seq$ be a sequence of nonnegative real numbers.  The following are equivalent.
\begin{enumerate}
\item{$\seq$ is a multiplier sequence}
\item{For each $n$, the polynomial $\ds T[(1+x)^n]:= \sum_{k=0}^{n}\binom{n}{k} \gamma_k x^k$ has only real zeros}
\item{The series $\ds \varphi(z) = \sum_{k=0}^{\infty} \frac{\gamma_k}{k!}z^k$ converges in the whole plane and either $\varphi(z)$ or $\varphi(-z)$ is of the form $\ds c e^{\sigma z}z^m \prod_{k=1}^{\omega}\left(1+ w_k{z}\right)$ where $c\in\mbb{R}$, $\sigma\geq 0$, $m$ is a nonnegative integer, $0\leq \omega\leq \infty$, $z_k>0$, and $\ds \sum_{k=1}^{\omega} {w^k}<\infty$}.
\end{enumerate}
\end{thm}

Similarly, corresponding to any sequence of real numbers $\seq$ one can define a linear operator $T_H$ on $\mbb{R}[x]$ by declaring $T_H[H_n(x)] = \gamma_n H_n(x)$ for all $n$, where $H_n(x)$ denotes the $n^{th}$ Hermite polynomial $\ds H_n(x) = (-1)^n e^{x^2} D^n e^{-x^2}$.  If the linear operator $T_H$ has the property that $T_H[p]$ has only real zeros whenever $p$ has only real zeros, then $\{\gamma_k\}_{k=0}^{\infty}$ is called an \textit{Hermite multiplier sequence}.  Examples of such sequences were demonstrated in the mid 1900's by Tur\'an \cite{Turan}, and also in 2001 by Bleecker and Csordas \cite{BC}.  In 2007, all such sequences were completely characterized by Piotrowski.  
\begin{thm}\label{PS}
\emph{(Piotrowski, Theorem 152 in \cite{andrzej})} Let $\seq$ be a sequence of nonnegative real numbers.  The following are equivalent.
\begin{enumerate}
\item{$\seq$ is a non-trivial Hermite multiplier sequence}
\item{$\seq$ is an nondecreasing multiplier sequence}
\item{The series $\ds \varphi(z) = \sum_{k=0}^{\infty} \frac{\gamma_k}{k!}z^k$ converges in the whole plane and either $\varphi(z)$ or $\varphi(-z)$ is of the form $\ds c e^{\sigma z}z^m \prod_{k=1}^{\omega}\left(1+\frac{z}{z_k}\right)$ where $c\in\mbb{R}$, $\sigma\geq 1$, $m$ is a nonnegative integer, $0\leq \omega\leq \infty$, $z_k>0$, and $\ds \sum_{k=1}^{\omega}\frac{1}{z^k}<\infty$}.
\end{enumerate}
\end{thm}

In this paper we investigate a related problem, where we use the generalized Laguerre polynomials in place of the Hermite polynomials.  To any sequence of real numbers $\seq$, one can define a linear operator $T_L$ on $\mbb{R}[x]$ by declaring $T_L\left[L_n^{(\alpha)}(x)\right] = \gamma_n L_n^{(\alpha)}(x)$ for all $n$, where $L_n^{(\alpha)}(x)$ denotes the $n^{th}$ Laguerre polynomial $ \ds L_n^{(\alpha)}(x) = \sum_{k=0}^{n}\binom{n+\alpha}{n-k}\frac{(-x)^k}{k!}$ and $\alpha>-1$.  If the linear operator $T_L$ has the property that $T_L[p]$ has only real zeros whenever $p$ has only real zeros, then $\{\gamma_k\}_{k=0}^{\infty}$ is called an \textit{$L^{(\alpha)}$-multiplier sequence}.  We opted not use the terminology ``Laguerre multiplier sequence," as this phrase has been used by other authors with a different meaning (See, for example, \cite{CCsurvey}). 

In a similar way, one can define $Q$-multiplier sequences, where $\ds Q=\{q_k\}_{k=0}^{\infty}$ is any simple polynomial set (i.e., $\deg(q_k)=k$ for each $k$).  Remarkably, every $Q$-multiplier sequence must be a (classical) multiplier sequence, regardless of the choice of $Q$.  In particular, the following result guarantees that every $L^{(\alpha)}$-multiplier sequence must also be a multiplier sequence.
\begin{thm}\label{thmqmsms}\emph{(Piotrowski, Theorem 158 in \cite{andrzej})}
Let $Q = \{q_k\}_{k=0}^{\infty}$ be a simple set of polynomials.  If the sequence $\{ \gamma_k\}_{k=0}^{\infty}$ is a $Q$-multiplier sequence, then the sequence $\{ \gamma_k\}_{k=0}^{\infty}$ is a multiplier sequence.
\end{thm}

In general, we will say that an operator $T$ \textit{preserves reality of zeros} if it has the property that $T[p]$ has only real zeros whenever $p$ has only real zeros.  Thus, a sequence is a multiplier sequence if its corresponding operator preserves reality of zeros.  Very recently, Borcea and Br\"and\'en gave a complete characterization of stability preserving operators.  A special case of their result is a  characterization of linear operators which preserve reality of zeros.  
\begin{thm}\label{BBthm}\emph{(Borcea-Br\"and\'en, Theorem 5 in \cite{BB})}
A linear operator $T:\R[x]\to \R[x]$ preserves reality of zeros if and only if either
\begin{enumerate}
\item{$T$ has range of dimension at most two and is of the form $T[f] = \alpha(f)P+ \beta(f)Q$ where $\alpha$ and $\beta$ are linear functionals on $\R[x]$ and $P$ and $Q$ are polynomials with only real interlacing zeros.}
\item{$\ds T[\exp(-xw)] = \sum_{n=0}^{\infty} \frac{(-w)^n T[x^n]}{n!}\in \overline{A}$, or }
\item{$\ds T[\exp(xw)] = \sum_{n=0}^{\infty} \frac{w^n T[x^n]}{n!}\in \overline{A}$,}
\end{enumerate}
where $\overline{A}$ denotes the set of entire functions in 2 variables that are limits, uniformly on compact subsets, of polynomials in the set 
$$
A  = \{f\in\R[x,w] \,\big|\,  f(x,w) \neq 0 \emph{ whenever } \emph{Im }x>0 \emph{ and } \emph{Im }w >0  \}.
$$
\end{thm}
With this characterization at hand, the crux of our problem is to find necessary and sufficient conditions on a sequence of real numbers under which the corresponding operator $T_L$ satisfies one of the conditions (1)-(3) above.  This task is quite difficult for a generic sequence, and as such we have not yet arrived at a complete characterization of 
$L^{(\alpha)}$-multiplier sequences. 
\newline \indent Finally, we note that throughout the paper we adopt the following convention: to avoid trivialities, we consider the identically zero function $f \equiv 0$ to have only real zeros, although this is clearly not the case.

\section{Trivial, Geometric, and Linear Sequences}

It is well known that the generalized Laguerre polynomials form an orthogonal set over the positive real axis with respect to the weight function $x^{\alpha}e^{-x}$ (recall that in this paper we are only considering  $\alpha>-1$).  Orthogonal polynomials have only simple real zeros. Furthermore, the zeros of consecutive polynomials in the sequence are interlacing.  As a consequence of this, it is easy to verify that any linear combination $aL_n^{(\alpha)}(x) + bL_{n+1}^{(\alpha)}(x)$ has only real zeros (one only needs to consider the sign of the individual terms and count zeros with the aid of the Mean Value Theorem).  We thus have the following fact:

\begin{prop}\label{trivial} Given $\gamma_{n}, \gamma_{n+1} \in \mathbb{R}$, any sequence of the form $( 0,0,\ldots, 0,0, \gamma_{n}, \gamma_{n+1}, 0,0,\ldots )$ is an $L^{(\alpha)}$-multiplier sequence.
\end{prop}
We will call sequences of the above form {\it trivial $L^{(\alpha)}$-multiplier sequences}. Unless stated otherwise, in what follows we only consider nontrivial $L^{(\alpha)}$-multiplier sequences. 

\subsection{Geometric $L^{(\alpha)}$-Multiplier Sequences}
We now consider the geometric sequences $\{ r^k \}_{k=0}^{\infty}$, $r \in \mathbb{R}$.  These sequences are (classical) multiplier sequences for all nonzero $r$ and are Hermite multiplier sequences if and only if $|r|\geq 1$. In contrast with these results, the only geometric sequence which is an $L^{(\alpha)}$-multiplier sequence is the constant sequence $\{1\}_{k=0}^{\infty}$.

\begin{prop} \label{r^k} The sequence $\ds \{ r^k \}_{k=0}^{\infty}$ is an $L^{(\alpha)}$-multiplier sequence if and only if $r=1$.
\end{prop}
\begin{proof} Consider the polynomial $p(x)=(x+b)^2$ for $b \in \R$. We can write $p(x)$ as
\[
p(x)=2 L_2^{(\alpha)}(x)-2(\alpha+2+b) L_1^{(\alpha)}(x)+(\alpha+b)^2+3\alpha+2b+2.
\] 
Applying the sequence $\{ r^k\}_{k=0}^{\infty}$ and then expanding in terms of the standard basis we obtain the polynomial
\begin{eqnarray*}
\bar p(x)&=&r^2 x^2+(2(\alpha+2+b)r-(2\alpha+4)r^2)x\\
&+&2+\alpha^2+2b+b^2+\alpha(3+2b)-2(2+\alpha+b)(1+\alpha)r+r^2(2+3\alpha+\alpha^2),
\end{eqnarray*}
with discriminant 
\[
\Delta=-4 r^2 (r-1)((2+\alpha)(1-r)+2b).
\]
From this representation we immediately see that $(i)$ if $r=1$ the discriminant is equal to zero and $(ii)$ large positive values (if $r >1$) or large negative values (if $r<1$) of $b$ result in a negative discriminant. This establishes the claim.  
\end{proof}

\subsection{Linear $L^{(\alpha)}$-Multiplier Sequences}
In \cite{andrzej} it is shown that for the simple Laguerre polynomials ($\alpha=0$) the sequence $\{a+k \}_{k=0}^{\infty}$ is not an $L^{(0)}$-multiplier sequence for $a >1 $ and $a < 0$ but it is an $L^{(0)}$-multiplier sequence for $a=1$ and $a=0$. The question whether $\{a+k\}_{k=0}^{\infty}$ is an $L^{(0)}$-multiplier sequence for $0< a < 1$ is left open. In this section we answer this question and completely characterize linear $L^{(\alpha)}$-multiplier sequences. 

\begin{lem} \label{a_alpha_1} $\{k+a\}_{k=0}^{\infty}$ is not an $L^{(\alpha)}$-multiplier sequence for any $\alpha$ if $a<0$.
\end{lem}
\begin{proof} The set $\left\{ L_k^{(\alpha)}(x) \right\}_{k=0}^{\infty}$ is a simple set of polynomials. Thus, by Theorem \ref{thmqmsms}, any sequence of real numbers $\{ \gamma_k \}_{k=0}^{\infty}$ that is an $L^{(\alpha)}$-multiplier sequence is a (classical) multiplier sequence. Since $\{ a+k\}_{k=0}^{\infty}$ is not a (classical) multiplier sequence for $a<0$ the result follows. 
\end{proof}

\begin{lem} \label{a_alpha_2} $\{k+a\}_{k=0}^{\infty}$ is not an $L^{(\alpha)}$-multiplier sequence if $a > \alpha+1$.
\end{lem}
\begin{proof} We recall that the polynomials $L_n^{(\alpha)}(x)$ satisfy the following ordinary differential equation:
\begin{equation} \label{LaguerreODE}
nL_n^{(\alpha)}(x)=(x-\alpha-1)L_n^{(\alpha)'}(x)-xL_n^{(\alpha)''}(x).
\end{equation}
It follows that 
\[
(a+k)L_k^{(\alpha)}(x)=aL_k^{(\alpha)}(x)+(x-\alpha-1)L_k^{(\alpha)'}(x)-xL_k^{(\alpha)''}(x).
\]
Thus the action of the sequence $\{ a+k\}_{k=0}^{\infty}$ on a polynomial is represented by the operator
\begin{equation}
\label{T}
T:=a+(x-\alpha-1)D-xD^2.
\end{equation}
Consider now the polynomial $(x+n)^n$, which clearly has only real zeros.  We have
\begin{eqnarray*}
T[(x+n)^n]&=&a(x+n)^n+(x-\alpha-1)n(x+n)^{n-1}-xn(n-1)(x+n)^{n-2}\\
&=&(x+n)^{n-2}[ a(x+n)^2+(x-(\alpha+1))n(x+n)-x(n^2-n)]\\
&=&(x+n)^{n-2}[x^2(a+n)+x(2an-n\alpha)+an^2-n^2(\alpha+1)].
\end{eqnarray*}
Calculating the discriminant of the polynomial in the square brackets we get
\begin{eqnarray*}
\Delta(n)&=&n^2[4a^2-4a\alpha+\alpha^2-4(a+n)(a-(\alpha+1))]\\
&=&n^2[\alpha^2+4a-4n(a-(\alpha+1))].
\end{eqnarray*}
It follows that if $a>(\alpha+1)$ then $\Delta(n) < 0$ for $n$ sufficiently large. Therefore $T[(x+n)^n]$ will have non-real zeros for large enough $n$. This completes the proof.
\end{proof}

\begin{lem} \label{a_alpha_3}
If $0\leq a\leq \alpha+1$, then $\{k+a\}_{k=0}^{\infty}$ is an $L^{(\alpha)}$-multiplier sequence. In particular, if $0 \leq a \leq 1$, then $\{k+a\}_{k=0}^{\infty}$ is an $L$-multiplier sequence.
\end{lem}

\begin{proof}
Consider the differential operator representation of the sequence at hand.
$$
T = (x-\alpha-1)D - x D^2 + a
$$
By the result of Borcea and Br\"and\'en (Theorem $\ref{BBthm}$) this operator preserves reality of zeros provided the polynomial
$$
a+(z-\alpha - 1)(-w) - z (-w)^2  = a-w(w+1)z+ w(\alpha+1)
$$
does not vanish whenever Im $z>0$ and Im $w>0$.  Setting the above equation equal to zero and solving for $z$ we obtain
$$
z = \frac{w(\alpha+1)+a}{w(w+1)}= (\alpha+1)\frac{w+w_0}{w(w+1)}, \qquad \left(w_0 = \frac{a}{\alpha+1}\right).
$$
Suppose Im $w>0$ and that $0\leq a \leq \alpha+1$.  Then $w_0<1$, and we have
$$
0< \arg(w)\leq \arg(w+w_0)\leq  \arg(w+1)< \pi,
$$
from which we obtain
$$
-\pi < -\arg(w+1) \leq \arg(w+w_0) - \arg(w) - \arg(w+1) \leq -\arg(w)< 0.
$$
Thus Im $z<0$ whenever Im $w>0$ and $0 \leq a \leq \alpha+1$. The proof is complete.    
\end{proof}

Combining lemmas $\ref{a_alpha_1}$, $\ref{a_alpha_2}$, and $\ref{a_alpha_3}$ we obtain the following theorem.

\begin{thm}\label{linchar}
$\{k+a\}_{k=0}^{\infty}$ is an $L^{(\alpha)}$-multiplier sequence if and only if $0\leq a\leq \alpha+1$.
\end{thm}

\section{The sequence $\{k(k-1)(k-2) \cdots (k-(n-1))  \}_{k=0}^{\infty}$} \label{arbdegree}
The purpose of this section is to prove that the above sequence is an $L^{(\alpha)}$-multiplier sequence for $\alpha >-1$ and $n \geq 1$. To establish this fact we need several auxiliary results. We begin with the following lemma.
\begin{lem} \label{commutator} Let $\delta$ be the operator defined by $\delta:=(x-(\alpha+1))D-xD^2$. Then for $k \geq 0$ we have
\[
[\delta, D^k]:=\delta D^k-D^k\delta=-k(1-D)D^k.
\]
\end{lem}
\begin{proof} If $k=0$ the result is trivial. Supposing the result holds for all integers up to $k$ we calculate
\begin{eqnarray*}
[\delta,D^{k+1}]&=&\delta D^{k+1}-D^{k+1} \delta=(\delta D^k)D-D(D^k \delta)\\
&=&(\delta D^k)D-D(\delta D^k+k(1-D)D^k)\\
&=&\delta D^{k+1}-(\delta D+(1-D)D)D^k-k(1-D)D^{k+1}\\
&=&-(k+1)(1-D)D^{k+1},
\end{eqnarray*}
establishing the desired equality.
\end{proof}

\begin{prop} \label{BIGlemma} Let $\delta$ be the operator defined by $\delta:=(x-(\alpha+1))D-xD^2$ and let $L_n^{(\alpha)}(x)$ be the $n^{th}$ generalized Laguerre polynomial. If
\begin{equation}\label{deltas}
\delta (\delta-1)(\delta-2)\cdots(\delta-(n-1))=\sum_{k=n}^{2n} {}_{2n}q_{k,\alpha}(x) D^k
\end{equation}
then
\begin{equation}\label{LaguerreRep}
\sum_{k=n}^{2n} {}_{2n}q_{k,\alpha}(x)z^k=n! (-1)^{n} z^n L_n^{(\alpha)}(x-xz).
\end{equation}
\end{prop}
\begin{proof} We proceed by induction on $n$. If $n=1$ the left hand side of $(\ref{deltas})$ is just $(x-(\alpha+1))D-xD^2$ which, after replacing $D^k$ by $z^k$ gives $(x-(\alpha+1))z-xz^2=-z L_1^{(\alpha)}(x-xz)$. Thus the statement of the proposition holds in case $n=1$. Next we calculate
\begin{eqnarray*}
 \delta(\delta-1)(\delta-2) \cdots (\delta-(n-1))&=&\sum_{k=n}^{2n}{}_{2n}q_{k,\alpha}(x)D^k \\
&=&\left(\sum_{k=n}^{2n-2} {}_{2n-2}q_{k,\alpha}(x)D^k \right)(\delta-(n-1))\\
&=&\sum_{k=n}^{2n-2} {}_{2n-2}q_{k,\alpha}(x)(x-(\alpha+1))D^{k+1}+\sum_{k=n}^{2n-2} {}_{2n-2}q_{k,\alpha}(x)kD^k\\
&-&\sum_{k=n}^{2n-2} {}_{2n-2}q_{k,\alpha}(x)kD^{k+1}-x\sum_{k=n}^{2n-2} {}_{2n-2}q_{k,\alpha}(x)D^{k+2}\\
&-&(n-1)\sum_{k=n}^{2n-2} {}_{2n-2}q_{k,\alpha}(x)D^k.
\end{eqnarray*}
Going from the second to the third line in this calculation we made use of Lemma \ref{commutator}. Replacing $D^k$ by $z^k$ in this expression along with the inductive hypothesis gives 
\begin{eqnarray*}
\sum_{k=n}^{2n} {}_{2n}q_{k,\alpha}(x)z^k&=&z(x-(\alpha+1))(n-1)!(-1)^{n-1}z^{n-1}L_{n-1}^{(\alpha)}(x-xz)\\
&+&(z-z^2)D_z\left[ (n-1)!(-1)^{n-1}z^{n-1}L_{n-1}^{(\alpha)}(x-xz) \right]\\
&-&xz^2(n-1)!(-1)^{n-1}z^{n-1}L_{n-1}^{(\alpha)}(x-xz)\\
&-&(n-1)(n-1)!(-1)^{n-1}z^{n-1}L_{n-1}^{(\alpha)}(x-xz)\\
&=&(n-1)!(-1)^{n-1}z^{n-1}\left\{z(x-(\alpha+1))L_{n-1}^{(\alpha)}(x-xz) \right.\\
&+& (1-z)\left[(n-1)L_{n-1}^{(\alpha)}(x-xz)-zx\frac{d}{dw}\left[L_{n-1}^{(\alpha)}(w) \right]_{w=x-xz} \right]\\
&-&\left. xz^2L_{n-1}^{(\alpha)}(x-xz)-(n-1)L_{n-1}^{(\alpha)}(x-xz)\right\}\\
&=&(n-1)!(-1)^{n-1}z^{n-1}\left\{-z(\alpha+n)L_{n-1}^{(\alpha)}(x-xz) \right.\\
&+& \left. z(x-xz)L_{n-1}^{(\alpha)}(x-xz)-z(x-xz)\frac{d}{dw}\left[L_{n-1}^{(\alpha)}(w) \right]_{w=x-xz} \right\}.
\end{eqnarray*}
Since the generalized Laguerre polynomials satisfy the relations 
\begin{eqnarray}
xDL_n^{(\alpha)}(x)&=&nL_n^{(\alpha)}(x)-(\alpha+n)L_{n-1}^{(\alpha)}(x) \label{recrel1} \\
DL_n^{(\alpha)}(x)&=&DL_{n-1}^{(\alpha)}(x)-L_{n-1}^{(\alpha)}(x) \label{recrel2}
\end{eqnarray}
(see for example Ch. 12 in \cite{Rainville}), it follows that 
\begin{eqnarray*}
&&(n-1)!(-1)^{n-1}z^{n-1}\left\{-z(\alpha+n)L_{n-1}^{(\alpha)}(x-xz) \right.\\
&+& \left. z(x-xz)L_{n-1}^{(\alpha)}(x-xz)-z(x-xz)\frac{d}{dw}\left[L_{n-1}^{(\alpha)}(w) \right]_{w=x-xz}\right\}\\
&=&(n-1)!(-1)^{n-1}z^{n-1}\left\{-znL_{n}^{(\alpha)}(x-xz) \right.\\
&+& z(x-xz)DL_{n-1}^{(\alpha)}(x-xz)-z(x-xz)L_{n-1}^{(\alpha)}(x-xz) \\
&+& \left. z(x-xz)L_{n-1}^{(\alpha)}(x-xz)-z(x-xz)DL_{n-1}^{(\alpha)}(x-xz)\right\}\\
&=&(n-1)!(-1)^{n-1}z^{n-1}(-znL_n^{(\alpha)}(x-xz))\\
&=&n!(-1)^nz^nL_n^{(\alpha)}(x-xz).
\end{eqnarray*}
The proof of Proposition \ref{BIGlemma} is complete.
\end{proof}

\begin{thm} The sequence $\set{k(k-1)(k-2)\cdots(k-(n-1))}_{k=0}^{\infty}$ is an $L^{(\alpha)}$-multiplier sequence for $n \in \N, \ n \geq 1$.
\end{thm}

\begin{proof}
Let $T$ be the linear operator defined by $T[L_k^{(\alpha)}(x)] = k(k-1)\cdots(k-n+1)L_k^{(\alpha)}(x)$.  Then $T = \delta (\delta-1)(\delta-2)\cdots(\delta-(n-1))$, where $\delta:=(x-(\alpha+1))D-xD^2$ and $D$ denotes differentiation with respect to $x$.  Using the definition of the generalized Laguerre polynomials, we have
\begin{eqnarray*}
n!(-1)^n z^n L_n^{(\alpha)}(x-xz) &=& n!(-1)^n z^n \sum_{k=0}^{n}\left(\begin{array}{c} n+\alpha \\ n-k \end{array} \right)(-1)^k \frac{(x-xz)^k}{k!}\\
&=& n!(-1)^n \sum_{k=0}^{n}\left(\begin{array}{c} n+\alpha \\ n-k \end{array} \right)(-1)^k \frac{x^k}{k!} z^n (1-z)^k.
\end{eqnarray*}
Thus, by Proposition \ref{BIGlemma},  
$$
T = n!(-1)^n \sum_{k=0}^{n}\left(\begin{array}{c} n+\alpha \\ n-k \end{array} \right)(-1)^k \frac{x^k}{k!} D^n (1-D)^k,
$$
and we have 
\begin{eqnarray*}
T[\exp(-xw)] &=& n!(-1)^n \sum_{k=0}^{n}\left(\begin{array}{c} n+\alpha \\ n-k \end{array} \right)(-1)^k \frac{x^k}{k!} D^n (1-D)^k [\exp(-xw)]\\
&=& n!(-1)^n \sum_{k=0}^{n}\left(\begin{array}{c} n+\alpha \\ n-k \end{array} \right)(-1)^k \frac{x^k}{k!} (-w)^n (1+w)^k \exp(-xw)\\
&=& n!(-1)^n (-w)^n L_n^{(\alpha)}(x+xw) \exp(-xw).
\end{eqnarray*}
Note that   
$$
f_m(x,w) = n!(-1)^n (-w)^n L_n^{(\alpha)}(x+xw) \left(1-\frac{xw}{m}\right)^m \qquad (m\in\N)
$$
converges uniformly on compact subsets to $T[\exp(-xw)]$ as $m\to\infty$.  Let $0<x_1<x_2<\cdots<x_n$ be the zeros of $L_n^{(\alpha)}(x)$ (recall that the generalized Laguerre polynomials have only real simple positive zeros).  Then $f_m(x,w)=0$ if and only if either $w=0$, $x(1+w)=x_k$ or $xw=m$, none of which occur when Im $x>0$ and Im $w>0$.  Therefore, by Theorem \ref{BBthm}, $T$ preserves reality of zeros.  
\end{proof}
We conclude this section with a corollary to this theorem. Although the corollary does not have a direct application to the development of $L^{(\alpha)}$-multiplier sequences, it is a quick result so we include it here.
\begin{cor} Let $\delta$ be as in Proposition \ref{BIGlemma} and let 
\[
\delta (\delta-1) \cdots (\delta-(n-1))=\sum_{k=n}^{2n} {}_{2n}q_{k,\alpha}(x)D^k.
\]
Then 
\[
\sum_{k=n}^{2n} {}_{2n}q_{k,\alpha}(x)=(-1)^n \prod_{k=1}^n (\alpha+k).
\]
\end{cor}
\begin{proof} By Proposition \ref{BIGlemma} we have
\[
\sum_{k=n}^{2n} {}_{2n}q_{k,\alpha}(x)=\sum_{k=n}^{2n} {}_{2n}q_{k,\alpha}(x)z^k \Big|_{z=1}=n!(-1)^nL_n^{(\alpha)}(0).
\]
On the other hand, using the generating function
\[
\frac{1}{(1-t)^{1+\alpha}} e^{\frac{-xt}{1-t}}=\sum_{n=0}^{\infty} L_n^{(\alpha)}(x)t^n
\]
we see that
\[
n!(-1)^nL_n^{(\alpha)}(0)=(-1)^n \prod_{k=1}^n (\alpha+k).
\]
\end{proof}

\section{Properties of $L^{(\alpha)}$-multiplier sequences} 
\subsection{Classical Properties}
There are a number of properties of the classical multiplier sequences which are easily verified.  Here we list those that carry over to $L^{(\alpha)}$-multiplier sequences. 
\begin{lem} \label{properties} Let $\set{\gamma_k}_{k=0}^{\infty}$ be an $L^{(\alpha)}$-multiplier sequence. Then:
\begin{itemize}
\item[(i)] If there exists an integers $n> m \geq 0$ such that $\gamma_m \neq 0$ and $\gamma_n=0$, then $\gamma_k=0$ for all $k \geq n$.
\item[(ii)] The terms of $\set{\gamma_k}_{k=0}^{\infty}$ are either all of the same sign, or they alternate in sign.
\item[(iii)] For any $r \in \R$, the sequence $\set{r \gamma_k}_{k=0}^{\infty}$ is also an $L^{(\alpha)}$-multiplier sequence.
\item[(iv)] The terms of $\set{\gamma_k}_{k=0}^{\infty}$ satisfy Tur\'an's inequality
\[
\gamma_k^2-\gamma_{k-1}\gamma_{k+1} \geq 0, \quad \quad k=1,2,3, \ldots
\] 
\end{itemize}
\end{lem}
\begin{proof} These claims follow immediately from Theorem \ref{thmqmsms} and the fact the generalized Laguerre polynomials form a simple set of polynomials.  Properties $(i)-(iv)$ for classical multiplier sequences have been established in \cite{Levin}.
\end{proof}
\begin{rmk} To draw further contrast between $L^{(\alpha)}$-multiplier sequences,  Hermite multiplier sequences, and classical multiplier sequences, we demonstrate that the following two properties, which hold for multiplier sequences and Hermite multiplier sequences, do \emph{not} hold for $L^{(\alpha)}$-multiplier sequences.  
\begin{itemize}
\item[(a)] If $\set{\gamma_k}_{k=0}^{\infty}$ is a multiplier sequence, then $\set{\gamma_k}_{k=m}^{\infty}$ is a multiplier sequence for any $m\in\mbb{N}$. 
\item[(b)] If $\set{\gamma_k}_{k=0}^{\infty}$ is a multiplier sequence, then $\set{(-1)^k\gamma_k}_{k=0}^{\infty}$ is a multiplier sequence.  
\end{itemize}
For property (a), we note that for the simple Laguerre polynomials ($\alpha=0$), the sequence $\{k+1\}_{k=0}^{\infty}$ is an $L^{(0)}$-multiplier sequence, but $\{k+1\}_{k=1}^{\infty} = \{k+2\}_{k=0}^{\infty}$ is not (see Theorem \ref{linchar}).

For property (b), we note again that $\{k+1\}_{k=0}^{\infty}$ is an $L^{(0)}$-multiplier sequence.  We now show that $\{(-1)^k(k+1)\}_{k=0}^{\infty}$ is not. The polynomial
$$
p(x) = (x-10)^2 = 82L_0^{(0)}(x)+16 L_1^{(0)}(x)+2 L_2^{(0)}(x)
$$   
has only real zeros, while
$$
3 \cdot 82L_0^{(0)}(x)-2\cdot16 L_1^{(0)}(x)+1\cdot2 L_2^{(0)}(x) = 3x^2+20x+56
$$
has two non-real zeros.
\end{rmk}

\begin{prop} Suppose that $\set{\gamma_n}$ is a non-trivial $L^{(\alpha)}$-multiplier sequence for some $\alpha >-1$. Then there exists an $m \in \mbb{Z}$ such that $\gamma_k=0$ for all $k<m$ and $\gamma_k \neq 0$ for all $k \geq m$. 
\end{prop}
\begin{proof} Since $\set{\gamma_n}$ is a non-trivial multiplier sequence, there is at least one $k \in \mbb{Z}$ such that $\gamma_k \neq 0$. Let $m$ be the minimal index such that $\gamma_m \neq 0$. It is easy to see that $\gamma_{m+1}$ and $\gamma_{m+2}$ are non-zero, for if either of them were zero, in light of Lemma \ref{properties} we would have to conclude that $\set{\gamma_k}_{k=0}^{\infty}$ is a trivial multiplier sequence. Suppose now that there exists a $n>m+2$ such that $\gamma_n=0$. By Lemma $\ref{bmax}$ (see below) there are constants $a_m, a_{m+2}$ such that the polynomial
\[
\widetilde{q(x)}=a_m\gamma_mL^{(\alpha)}_m(x)+a_{m+2}\gamma_{m+2} L^{(\alpha)}_{m+2}(x)
\]
has some non-real zeros.  On the other hand, by Lemma $\ref{open}$ then there exists $a_n$ such that 
\[
q(x)=a_mL^{(\alpha)}_m(x)+a_{m+2}L^{(\alpha)}_{m+2}(x)+a_nL^{(\alpha)}_n(x) = a_n\left(L^{(\alpha)}_n(x) +\frac{a_m}{a_n}L^{(\alpha)}_m(x)+\frac{a_{m+2}}{a_n}L^{(\alpha)}_{m+2}(x)\right)  
\]
has only real zeros.  Applying the $L^{(\alpha)}$-multiplier sequence $\set{\gamma_k}_{k=0}^{\infty}$ to $q(x)$ we obtain the polynomial $\widetilde{q(x)}$, a contradiction. Hence $\gamma_k \neq 0$ for all $k \geq m$ and the proof is complete. 
\end{proof}

\subsection{Monotonicity of $L^{(\alpha)}$-multiplier sequences}
The main result in this section is that if a classical multiplier sequence is an $L^{(\alpha)}$-multiplier sequence, then it must be non-decreasing. We note that an analogous statement is true for the Hermite multiplier sequences. The converse is also true for Hermite multiplier sequences, but not for $L^{(\alpha)}$-multiplier sequences (recall the sequences $\{r^k\}$ for $r>1$!). We next lay the necessary groundwork to establish the stated monotonicity result for $L^{(\alpha)}$-multiplier sequences.
\newline \indent We begin with two simple, but very useful lemmas. The first one essentially says that if a polynomial has only simple real zeros and one makes a small perturbation of the coefficients, then the resulting polynomial also has only real zeros. 
\begin{lem}\label{open}
Let $p$ and $q$ be real polynomials and suppose $\deg(q)<\deg(p)$.  If $p$ has only simple real zeros then there exists $\epsilon>0$ such that $p(x)+ b q(x)$ has only real zeros whenever $|b|<\epsilon$. 
\end{lem}

\begin{proof}
Suppose no such $\epsilon$ exists.  Then we can obtain a sequence of real numbers $\{b_n\}_{n=1}^{\infty}$ converging to zero such that, for each $n$, the real polynomial $p_n(x) = p(x)+ b_n q(x)$ has some non-real zeros.  The polynomials $p_n$ converge uniformly on compact subsets of $\mbb{C}$ to $p$.  By Hurwitz' Theorem, the zeros of $p$ must be limits of the zeros of $p_n$, contradicting the fact that the zeros of $p$ are all real and simple. 
\end{proof}



The next result is similar in nature.  If we begin with a polynomial which has some non-real zeros then any small perturbation of the coefficients will result in another polynomial which has some non-real zeros.  

\begin{lem}\label{closed}
Let $p$ and $q$ be real polynomials and suppose $\deg(q)<\deg(p)$.  If $p$ has some non-real zeros then there exists $\epsilon>0$ such that $p(x)+ b q(x)$ has some non-real zeros whenever $|b|<\epsilon$. 
\end{lem}

\begin{proof}
We appeal to Hurwitz' Theorem once again.  If no such $\epsilon$ exists,  then we can obtain a sequence of real numbers $\{b_n\}_{n=1}^{\infty}$ converging to zero such that, for each $n$, the real polynomial $p_n(x) = p(x)+ b_n q(x)$ has only real zeros.  The polynomials $p_n$ converge uniformly on compact subsets of $\mbb{C}$ to $p$.  By Hurwitz' Theorem, the zeros of $p$ must be limits of the zeros of $p_n$, but non-real numbers are never the limit of a sequence of real numbers, a contradiction.  
\end{proof}

\begin{lem}\label{bmax}
For $n\geq 2$ and $b\in\mbb{R}$, define
\begin{eqnarray*}
f_{n, b, \alpha} (x)&:=&L_n^{(\alpha)}(x)+b L_{n-2}^{(\alpha)}(x), \quad \textrm{and}\\
E_n&:=&\set{ b \in \R \ | \ f_{n, b, \alpha}(x) \quad \textrm{has only real zeros}}.
\end{eqnarray*}
Then $\max(E_n)$ exists, and is a positive real number.
\end{lem}

\begin{proof}
By Lemma $\ref{open}$, there exists $\epsilon>0$ such that $(-\epsilon, \epsilon)\subseteq E_n$.  In particular, $E_n$ is nonempty and $\max(E_n)$, if it exists, is positive.  It now suffices to show that $E_n$ is closed and bounded above.  

Suppose $t \in (\mbb{R}\setminus E_n)$.  Then, by Lemma $\ref{closed}$, there exists $\delta>0$ such that 
$$
f_{n, t, \alpha}(x)+ b L_{n-2}^{(\alpha)}(x)  = L_n^{(\alpha)}(x)+(t+b) L_{n-2}^{(\alpha)}(x). 
$$ 
has non-real zeros whenever $|b|<\delta$.  That is to say, $(b-\delta, b+\delta)\subseteq (\mbb{R}\setminus E_n)$.  Whence, $\mbb{R}\setminus E_n$ is open and, therefore, $E_n$ is closed.

To show that $E_n$ is bounded above, we consider the $(n-2)^{nd}$ derivative of $f_{n,b,\alpha}$.  A calculation shows
$$
\frac{d^{n-2}}{dx^{n-2}} f_{n, b, \alpha} (x) = \frac{1}{2}x^2 - (n+\alpha) x + \frac{(n+\alpha)(n+\alpha-1)}{2} + b.  
$$
Thus $\ds \frac{d^{n-2}}{dx^{n-2}} f_{n, b, \alpha} (x)$, and therefore $ f_{n, b, \alpha} (x)$, have some non-real zeros whenever $b$ is sufficiently large.
\end{proof}

\begin{thm}  If the sequence of positive real numbers $\ds{\set{\gamma_k}_{k=0}^{\infty}}$ is a non-trivial $L^{(\alpha)}$-multiplier sequence, then $\gamma_k \leq \gamma_{k+1}$ for all $k \geq 0$.
\end{thm}
\begin{proof} Let $T_L$ denote the operator associated to the $\ds{L^{(\alpha)}}$-multiplier sequence $\set{\gamma_k}_{k=0}^{\infty}$. With the notation of Lemma \ref{bmax}, for each $n \geq 2$ the function
\[
f_{n,\beta_n^*,\alpha}(x)=L_n^{(\alpha)}(x)+\beta^*_n L_{n-2}^{(\alpha)}(x) \qquad (\beta_n^* = \max(E_n))
\]
has only real zeros. It follows that 
\[
T_L[f_{n,\beta_n^*,\alpha}(x)]=\gamma_n L_n^{(\alpha)}(x)+\gamma_{n-2}\beta^*_n L_{n-2}^{(\alpha)}(x)=\gamma_{n} \left(L_n^{(\alpha)}(x)+\frac{\gamma_{n-2}}{\gamma_n}\beta^*_n L_{n-2}^{(\alpha)}(x) \right)
\]
also only has real zeros. By Lemma \ref{bmax} we must have $\ds{\frac{\gamma_{n-2}}{\gamma_n}\beta^*_n \leq \beta_n^*}$ which gives $\ds{0 < \frac{\gamma_{n-2}}{\gamma_n} \leq 1}$. On the other hand, by Lemma \ref{properties}, we have
\[
\gamma_{n-1}^2-\gamma_{n}\gamma_{n-2} \geq 0, \quad \quad (n \geq 2)
\]
which means $\ds{\left( \frac{\gamma_{n-1}}{\gamma_{n-2}}\right)^2 \geq \frac{\gamma_{n}}{\gamma_{n-2}}\geq1}$. In other words $\gamma_{n-1} \geq \gamma_{n-2}$ and the proof is complete.
\end{proof}

\section{Open questions}
Contrary to the linear sequences, quadratic (and higher degree) multiplier sequences for generalized Laguerre bases are not well understood and are far from being completely characterized. Recall from Section \ref{arbdegree} that $L^{(\alpha)}$-multiplier sequences of arbitrary degrees exist. As a result, investigations into quadratic, cubic, and higher degree $L^{(\alpha)}$-multiplier sequences are not vacuous, and rather challenging. One of the reasons for this is that although one can naturally get higher order $L^{(\alpha)}$-multiplier sequences from lower order ones, one can not get them all this way. There are for example quadratic $L^{(\alpha)}$-multiplier sequences that do not factor as a product of the linear ones (in the differential operator sense). In this section we present some partial results in the characterization of sequences of the form $\{k^2+ak+b \}_{k=0}^{\infty}$ for the simple Laguerre polynomials ($\alpha=0$) and pose some open questions.
\newline
Based on some partial results, we believe that the following conjecture is true:
\begin{conj} The sequence $\{k^2+ak+b\}_{k=0}^{\infty}$ is an $L^{(0)}$-multiplier sequence if and only if 
$$
-1\leq a \leq 3  \quad \text{   and   }\quad  \max \{ 0, a-1 \} \leq b \leq \ds{\frac{1}{8}(1+a)^2}.
$$
\end{conj}
It is easy to show that if $\{k^2+ak+b\}_{k=0}^{\infty}$ is an $L^{(0)}$-multiplier sequence then then $a\geq -1$ and $\ds 0\leq b \leq \frac{1}{4} (a+1)^2$. It is a bit more involved to improve the upper bound on $b$ to $\frac{1}{8} (a+1)^2$ but it can be done by using the result of Borcea and Br\"and\'en. The proof involves the verification of stability of a certain polynomial in two complex variables, and is a bit technical. We believe that for the characterization of polynomial type $L^{(\alpha)}$-multiplier sequences of arbitrary (fixed) degree, additional techniques will be needed. 
Using a theorem due to Newton, we can easily  establish the bounds $a \leq 4$ and $a-1\leq b$ and another application of Borcea Br\"and\'en gives that if $1 \leq a \leq 3$ then $b=a-1$ is allowed, in other words $\{k^2+ak+a-1 \}$ is an $L^{(0)}$-multiplier sequence. Though these results pointed to the formulation of the above conjecture, we were unable to prove the result so far and the question remains open. The situation is similar for all polynomial sequences of degree 3 or higher.

\end{document}